\documentclass[12pt]{article}

\usepackage[margin=1in]{geometry}  
\usepackage{amsmath}               
\usepackage{amsthm}                
\usepackage{amssymb}
\usepackage[colorlinks=true]{hyperref}
\usepackage{mathabx}
\usepackage[all]{xy}
\usepackage{amsmath,amssymb,mathabx}
\usepackage{enumerate}

\newtheorem{thm}{Theorem}[section]
\newtheorem{lem}[thm]{Lemma}
\newtheorem{prop}[thm]{Proposition}

\newtheorem{conj}[thm]{Conjecture}
\theoremstyle{definition}
\newtheorem{rmk}{Remark}
\newtheorem{defn}{Definition}
\newtheorem{expl}{Example}
\allowdisplaybreaks[1]

\providecommand{\coker}{\operatorname{coker}}
\providecommand{\Ext}{\operatorname{Ext}}
\providecommand{\holim}{\operatorname{holim}}

\providecommand{\deg}{\operatorname{deg}}
\providecommand{\ot}{\leftarrow}
\providecommand{\psphere}{S_{(p)}}
\providecommand{\ZLOC}{\mathbb Z_{(3)}}
\newcommand{\Z}{\mathbb Z}
\newcommand{\N}{\mathbb N}

\newcommand{\threesphere}{S_{(3)}}

\providecommand{\LMAX}{\ell^m_0}
\providecommand{\LMAXB}{\ell^{m}_1}
\providecommand{\HOT}{\textrm{higher order terms}}

\providecommand{\BMF}{\mathcal B_{MF}}
\providecommand{\floor}[1]{\left\lfloor#1
\right\rfloor}
\providecommand{\ceil}[1]{\left\lceil#1
\right\rceil}

\begin{document}


\title{The Adams-Novikov $E_2$-term for Behrens'
spectrum $Q(2)$ at the prime 3}

\author{Donald M.\ Larson}

\maketitle

\begin{abstract}
We compute 
the Adams-Novikov $E_2$-term of a spectrum $Q(2)$ constructed by M.\ Behrens.  
The homotopy groups of $Q(2)$ are closely tied to the
3-primary stable homotopy groups of spheres; in particular, 
they are conjectured to detect the homotopy beta family
of Greek letter elements at the prime 3.
Our computation leverages techniques used 
by Behrens to compute the rational homotopy of 
$Q(2)$, and leads to a conjecture 
that the Adams-Novikov $E_2$-term for $Q(2)$ 
detects the algebraic beta family in the $BP$-based Adams-Novikov 
$E_2$-term for the 3-local sphere.
\end{abstract}

\section{Introduction}\label{Intro}

Among the central objects of study in
stable homotopy theory are the $p$-local
stable homotopy groups of spheres for a
prime $p$, i.e., the homotopy groups of the 
$p$-local sphere $S_{(p)}$. The chromatic 
convergence theorem of Hopkins and Ravenel 
\cite{Rav:NP} says that 
further 
localizations of $\psphere$ 
yield successive 
approximations of $\pi_*\psphere$; more precisely, 
if $L_n$ is localization with respect to the
Johnson-Wilson spectrum $E(n)$ at $p$ \cite{JW1}, then
\[
\psphere\simeq\holim(L_0\psphere\ot L_1\psphere
\ot L_2\psphere\ot\cdots).
\]
For each $n$, $L_n\psphere$ lies in a homotopy fracture 
square with $L_{K(n)}\psphere$ \cite{Hov},
where $L_{K(n)}$ is localization
with respect to the $n$th Morava $K$-theory spectrum $K(n)$
at $p$. This means the groups
$\pi_*L_{K(n)}\psphere$ for $n\geq0$
are building blocks for $\pi_*\psphere$.
The spectra $L_{K(n)}\psphere$ are
the {\em $K(n)$-local spheres} and $n$
is the {\em chromatic level}.  
The spectrum $Q(2)$ that we study in this paper 
yields information at chromatic level 2.

Indeed, Behrens \cite{Beh:Mod} constructs $Q(2)$
in an effort to reinterpret previous groundbreaking
work (\cite{Shimo:L2S0p3},\cite{GHMR},\cite{MahRezk}) on 
$\pi_*L_{K(2)}\threesphere$---groups which lie at 
the edge of what is accessible computationally, as very little is known about 
the $K(n)$-local sphere 
at any prime 
for $n\geq3$. The spectrum $Q(2)$ is an 
$E_{\infty}$ ring spectrum with the property that 
\begin{equation}\label{cofiber}
DL_{K(2)}Q(2)\xrightarrow{D\eta} 
L_{K(2)}\threesphere\xrightarrow{\eta} L_{K(2)}Q(2)
\end{equation}
is a cofiber sequence, where $\eta$ is the $K(2)$-localized
unit map of $Q(2)$ and $D$ is the 
$K(2)$-local Spanier-Whitehead duality functor.  
The ``2'' in the notation reflects the
fact that $Q(2)$ is built using degree
2 isogenies of elliptic curves (see
Subsection \ref{maps}) and is not a reference to 
the chromatic level $n=2$.  The sequence
(\ref{cofiber}) implies that 
$L_{K(2)}\threesphere$ is built from
$L_{K(2)}Q(2)$ and $DL_{K(2)}Q(2)$ and that their respective
homotopy groups lie in a long exact sequence.
Behrens (\cite{Beh:Mod}, Section 1.4.2)
observes that there is a spectral sequence
converging to $\pi_*Q(2)$ whose input
is the cohomology of the totalization of a double 
cochain complex $C^{*,*}$:
\begin{equation}\label{spectralsequence}
E_2^{s,t}Q(2):=H^{s,t}(\operatorname{Tot}C^{*,*})
\Rightarrow\pi_{2t-s}Q(2).
\end{equation}
This is the Adams-Novikov spectral sequence for $Q(2)$. 
In particular, the Adams-Novikov 
$E_2$-term for $Q(2)$ is itself
computable via a double complex spectral sequence.

In this paper we compute the double complex
spectral sequence converging to the 
Adams-Novikov $E_2$-term for $Q(2)$,
thereby obtaining explicit descriptions of 
the elements in this $E_2$-term 
up to an ambiguity in two torsion 
$\ZLOC$-submodules which we denote 
$U^{1,*}\subset E_2^{1,*}Q(2)$ and 
$U^{2,*}\subset E_2^{2,*}Q(2)$. 
The double complex $C^{*,*}$ 
is built from the cobar resolution
of an elliptic curve Hopf algebroid
$(B,\Gamma)$ over $\ZLOC$ to be defined
in Section \ref{ANSS}.  
Throughout this paper, $\Ext^{*,*}$
(or just $\Ext^*$)
will denote the Hopf algebroid cohomology
of $(B,\Gamma)$, i.e., 
\[
\Ext^{*,*}:=\Ext^{*,*}_{\Gamma}(B,B)
\]
in the category of $\Gamma$-comodules,
and $\nu_p(x)$
will denote the $p$-adic valuation of a
($p$-local) integer $x$.
The following is our main theorem.
\begin{thm}\label{main}
The Adams-Novikov $E_2$-term for
$Q(2)$ is given by
\begin{align*}
E_2^{0,t}Q(2)&=\begin{cases}
\ZLOC, & t=0,\\
0, & t\neq0,
\end{cases}\\
\, &\, \\
E_2^{1,t}Q(2)&=\begin{cases}
{\displaystyle\bigoplus_{n\in\N}\ZLOC}, & t=0,\\
\Z/(3)\oplus\Z/(3), & t=4,\\
\Z/(3^{\nu_3(3m)}), & t=4m, m\geq2,\\
U^{1,t}, & t=4m+2, m\geq1,
m\equiv13\bmod 27, \\
\Z/(3^{\nu_3(6m+3)}), & t=4m+2, m\geq1,
m\nequiv13\bmod 27,\\
0, & \text{otherwise},
\end{cases}\\
\, &\, \\
E_2^{2,t}Q(2)&=\Ext^{2,t}\oplus\Ext^{1,t}
\oplus M
\end{align*}
where
\[
M=\begin{cases}
{\displaystyle
\bigoplus_{n\in\N}\Z/3^{\nu_3(6m+3)}
}, & t=4m+2, m\leq-1,\\
U^{2,t}\oplus\left({\displaystyle
\bigoplus_{n\in\N}\Z/3^{\nu_3(6m+3)}
}\right),& 
t=4m+2, m\geq1,
m\equiv13\bmod27,\\
{\displaystyle
\bigoplus_{n\in\N}\Z/3^{\nu_3(6m+3)}
}, & t=4m+2, m\geq1, m\nequiv13\bmod 27,\\
0, & \text{otherwise},
\end{cases}
\]
and $E_2^{s,t}Q(2)=\Ext^{s,t}\oplus\Ext^{s-1,t}$
for $s\geq3$.
\end{thm}   
The cohomology groups $\Ext^*$ 
have been computed by
Hopkins and Miller \cite{Bauer}
and they appear as summands in the
Adams-Novikov $E_2$-term for $Q(2)$ by virtue of how
$Q(2)$ is constructed (see Section 
\ref{ANSS}).  

In addition to being a concrete computational
tool for accessing $\pi_*Q(2)$,
Theorem \ref{main} also sheds light on
a conjectured relationship between $Q(2)$ and
the beta family in the 3-primary stable stems, 
as follows.
The spectrum $Q(2)$ is a special case
of a more general object $Q(N)$, 
an $E(2)$-local ring spectrum at the prime $p$
built from degree $N$ isogenies,
that exists as long as $p$ does not
divide $N$.
Behrens (\cite{Beh:Cong}, Theorem 12.1)
proves that for $p>3$ and $N$ a topological 
generator of $\mathbb Z_p^{\times}$, 
nontrivial homotopy divided beta family
elements $\beta^h_{i/j,k}\in\pi_*\psphere$
(which are seen in $\pi_*L_2\psphere$)
are detected by the homomorphism 
$(\eta_{E(2)})_*$
induced by the $E(2)$-localized unit map
$\eta_{E(2)}:L_2S_{(p)}\to Q(N)$.  
Behrens conjectures that this holds
for $p=3$ and all corresponding $N$ (\cite{Beh:Cong}, Section 1).
The case $p=2$ is addressed by Behrens
and K.\ Ormsby in \cite{Behrens:Ormsby}.

The algebraic divided beta family lives on the 2-line
of the $BP$-based Adams-Novikov $E_2$-term
for the $p$-local sphere and comprises elements
\[
\beta^a_{i/j,k}\in\Ext^{2,2i(p^2-1)-2j(p-1)
}_{BP_*BP}(BP_*,BP_*)
\]
for 
certain $i,j,k\in\Z$ \cite{MRW}.  
These elements also appear in the
$E_2$-term for $L_2\psphere$.
Our computation
yields evidence for an algebraic version of
Behrens' conjecture 
in the case $p=3$ and $N=2$.  
\begin{conj}\label{conjecture}
The 
elements $\beta^a_{i/j,k}$
have nontrivial image under the map 
of Adams-Novikov $E_2$-terms induced by 
$\eta_{E(2)}:L_{2}\threesphere\to Q(2)$.
\end{conj} 

\begin{rmk}
The statement of Theorem \ref{main} reveals
that the undetermined submodules $U^{1,*}$
and $U^{2,*}$ together constitute a small sliver of 
the Adams-Novikov $E_2$-term for $Q(2)$.
In particular, these submodules
could not possibly contain all of the
algebraic divided beta family.
\end{rmk}

In Section \ref{outline} we outline our main results
that lead to Theorem \ref{main}.
In Section \ref{ANSS} we recall the construction
of $Q(2)$ and the algebraic underpinnings
of the double complex spectral sequence
for $C^{*,*}$.  Sections \ref{heart},
\ref{connecting}, and \ref{deetwodiff} are 
the technical heart of the paper,
where we prove the results stated in
Section \ref{outline}.
We conclude in Section \ref{Greek}
with evidence that Conjecture \ref{conjecture} holds.


The author would like to thank Doug Ravenel for
his unwavering support and encouragement.  
Thanks also go to Mark Behrens for invaluable assistance, 
especially during the author's visit to MIT
in April 2011.  Finally, we thank Mark Johnson
for many helpful comments on earlier
versions of this paper, as well as the anonymous
referee for numerous insights, both stylistic
and mathematical.

\section{Statement of main results}
\label{outline}

In this section we state the results that 
constitute our proof of Theorem \ref{main}
(largely suppressing the $t$-degree throughout for readability).   
Our approach is based 
on previous work of Behrens on $\pi_*Q(2)\otimes\mathbb Q$ \cite{Behrens:Rational}.

The first of our results reduces the double complex spectral 
sequence for $C^{*,*}$ to the cohomology of a 
singly-graded three-term cochain complex and the 
computation of one additional nontrivial differential.  
Recall from Section \ref{Intro} that $\Ext^*$ is 
the cohomology of the elliptic curve Hopf algebroid 
$(B,\Gamma)$.  
\begin{prop}\label{reduction}
In the double complex spectral sequence
for $C^{*,*}$, there are only two nontrivial 
$E_1$-page differentials given by $\ZLOC$-module maps
\begin{equation}\label{threecomplex}
\Ext^0\xrightarrow{\Phi}\Ext^0\oplus B\xrightarrow{\Psi}B,
\end{equation}
there is only one nontrivial $E_2$-page differential 
\[
\widetilde d:\Ext^1\to\coker\Psi,
\]
and $E_3=E_{\infty}$.  Moreover,
\[
H^s(\operatorname{Tot}C^{*,*})=
\begin{cases}
H^0C^{*,0},&s=0\\
H^1C^{*,0}\oplus\ker\widetilde d,&s=1\\
\coker\widetilde d\oplus\Ext^2\oplus\Ext^1,&s=2\\
\Ext^s\oplus\Ext^{s-1},&s\geq3.
\end{cases}
\]
\end{prop}
By a slight abuse of notation, we will denote
the cochain complex (\ref{threecomplex})
by $C^*$ (so that $C^*=C^{*,0}$).  The following proposition describes
a two-stage filtration that we use to compute $H^*C^*$.

\begin{prop}\label{filtration}
There is a filtration 
$C^*=F^0\supset F^1\supset F^2$
of $C^*$ inducing a short exact sequence 
$0\to C' \to C^*\to C''\to0$, where 
\[
C'=(0\to B\xrightarrow{h}B),\quad
C''=(\Ext^0\xrightarrow{g}\Ext^0\to 0)
\]
The resulting long exact sequence in cohomology is
\[
0\to H^0C^*\to\ker g\xrightarrow{\delta^0}\ker h
\to H^1C^*\to\coker g\xrightarrow{\delta^1}\coker h\to H^2C^*\to 0
\]
so that
$H^0C^*=\ker\delta^0$, $H^2C^*=\coker\delta^1$, and
$H^1C^*$ lies in the short exact sequence
\begin{equation} \label{hone}
0\to\coker\delta^0\to H^1C^*\to\ker\delta^1\to0.
\end{equation}

\end{prop}

Proposition \ref{filtration} shows that computing
$H^*C^*$ via this two-stage filtration starts with 
the kernels and cokernels of the maps 
$g:\Ext^0\to\Ext^0$ and $h:B\to B$.  
We compute these kernels and cokernels using
judicious choices of bases for $B$ and $\Ext^0$
as $\ZLOC$-modules.
\begin{prop}\label{gandh}
As modules over $\ZLOC$,
${\displaystyle\ker g=\ker h=\bigoplus_{n\in\mathbb N}\ZLOC}$,
\[
\coker g=
\left(\bigoplus_{n\in\mathbb N}\ZLOC
\right)\oplus
\left(
\bigoplus_{x\in\mathcal B_{MF}^{\neq0}}\Z/(3^{\nu_3(\deg x)+1})
\right)
\]
where $\mathcal B_{MF}^{\neq0}$ is a basis for the submodule
of $\Ext^{0,*}$ of elements of nonzero $t$-degree
(see Definition \ref{MFBASIS}), and
\[
\coker h=\bigoplus_{i<j\in\mathbb Z}
\left(
\mathbb Z/(3^{\nu_3(i+j)+1})\oplus
\mathbb Z/(3^{\nu_3(2i+2j+1)+1})
\right).
\]
\end{prop}

The following theorem describes $H^*C^*$.  We
prove this result by computing the connecting homomorphisms
$\delta^0$ and $\delta^1$ from Proposition \ref{filtration}.
As in Proposition \ref{gandh}, the proof is based on judicious
choices of generators for the sources and targets.
\begin{thm}\label{cohomology} 
\begin{enumerate}[(a)]
\item $H^0C^*=\ZLOC$
\item 
${\displaystyle H^1C^*=
\left(\bigoplus_{n\in\N}\ZLOC\right)
\oplus\left(
\bigoplus_{m>0}\Z/(3^{\nu_3(3m)})
\right)}$\\
$\displaystyle{\qquad\qquad\qquad\qquad\quad\oplus\left(
\bigoplus_{
\substack{
m>0\\m\nequiv13\bmod27
}}\Z/(3^{\nu_3(6m+3)})
\right)\oplus U^1}$
\item 
${\displaystyle
H^2C^*=\left(\bigoplus_{m\neq0}\bigoplus_{n\in\N}
\Z/(3^{\nu_3(6m+3)})\right)\oplus
U^2}$
\end{enumerate}
\end{thm}
Finally, we compute the differential $\widetilde d$
described in Proposition \ref{reduction}.
\begin{thm}\label{dtwodiff}
$\ker\widetilde d=\mathbb Z/3$,
and \, 
${\displaystyle
\coker\widetilde d =H^2C^*\bigg/
\left(
\bigoplus_{m\in3\Z}\Z/(3^{\nu_3(6m+3)})
\right)}$.
\end{thm}
Our computation will reveal that
generators of $\ZLOC$ summands in $E_2^{s,t}Q(2)$
lie in $t$-degree 0, 
generators of $\Z/(3^{\nu_3(3m)})$ summands
lie in $t$-degree $4m$, and 
generators of $\Z/(3^{\nu_3(6m+3)})$
summands lie in $t$-degree $4m+2$ (see, e.g.,
Remark \ref{tdegrees}).  With this,
Proposition \ref{reduction} and Theorems
\ref{cohomology} and \ref{dtwodiff} piece together
to yield Theorem \ref{main}.

\begin{rmk}\label{examplecase}
As we shall see in Section \ref{connecting},
the summand
\begin{equation}\label{examplesummand}
\bigoplus_{
       \substack{
       m>0\\m\nequiv13\bmod27
       }}\Z/(3^{\nu_3(6m+3)})
\end{equation}
of 
$E_2^{1,*}Q(2)$
is a submodule of $\ker\delta^1$. 
The source and target of $(\delta^1)_{4m+2}$
(the restriction of $\delta^1$ to elements of 
$t$-degree $4m+2$ in $\coker g$)
are infinite direct sums of copies of $\Z/(3^{\nu_3(6m+3)})$.  
If $m\nequiv13\bmod27$ then $\Z/(3^{\nu_3(6m+3)})\cong\Z/(3)$, 
$\Z/(9)$, or $\Z/(27)$, making the kernel and cokernel of $(\delta^1)_{4m+2}$ explicitly computable in those
cases.  This is not true for
$m\equiv13\bmod27$, for which $\nu_3(6m+3)\geq4$.  
To demonstrate this, below we have the first few columns of
matrix representations of 
$(\delta^1)_{4m+2}$
for a general $m\nequiv13\bmod27$, and $m=13$
(so that $\Z/(3^{\nu_3(6m+3)})\cong\Z/(81)$), respectively:
\[
\left[\begin{array}{cccccc}
\vdots & \vdots & \, & \, & \, & \,\\
u_0 & * & \, & \, & \, & \,\\
0 & u_1 & \, & \vdots & {\pmb\vdots} & \,\\
\vdots & 0 & \ddots & * & {\pmb 0} & \,\\
\, & \vdots & \, & u_y & {\pmb 0}& \cdots\\
\, & \, & \, & 0 & {\pmb 0}& \,\\
\, & \, & \, & \vdots & {\pmb\vdots}& \,\\
\end{array}
\right],\quad
\left[\begin{array}{cccccc}
0 & 78 & 31 & 39 & {\pmb0} & \, \\
8 & 21 & 17 & 72 & {\pmb0} & \,\\
80 & 5 & 79 & 6 & {\pmb0} & \,\\
0 & 62 & 56 & 19 & {\pmb{27}}& \, \\
0 & 0 & 44 & 72 & {\pmb0}& \, \\
0 & 0 & 0 & 26 & {\pmb{54}}& \cdots \\
0 & 0 & 0 & 0 & {\pmb0}& \,\\
\vdots & \vdots & \vdots & \vdots & {\pmb\vdots}& \,\\
\end{array}
\right]
\]
(see Eq.\ \eqref{MATRIX2} in 
Subsection \ref{deltaonethree}).
The entries $u_0,\ldots,u_y$ are units.
The bolded zero column on the left yields
the summand \eqref{examplesummand}
and also makes the cokernel a direct
sum of cyclic $\ZLOC$-modules when $m\nequiv13\bmod27$.  
The corresponding column for $m\equiv13\bmod27$
is always {\em nonzero} (see Lemma \ref{lemcombo}(a)),
as in the $m=13$ example shown above on the right.
This causes complications, including relations
in $\coker\delta^1$ that we cannot compute in general.
It is precisely 
these unknown parts of $\ker\delta^1$ and $\coker\delta^1$
that constitute $U^1$ and $U^2$, respectively (see
Definition \ref{unknowns}).
In this $m=13$ example, 
the ordered bases for the source and target
are $\{D_0^{13},D_1^{13},D_2^{13},\ldots\}$
and $\{B_0^{13},B_1^{13},B_2^{13},\ldots\}$
(see Definitions \ref{defnAB} and 
\ref{defnCD}), the
kernel is $\Z/(81)$ generated by 
$-27D_3^{13}+D_4^{13}$,
and the cokernel is
\begin{align}\label{specialcokernel}
\begin{split}
\left(\Z/(81)\{B_0^{13},B_1^{13}\}/
(B_0^{13}-3B_1^{13}=0)\right)
&\oplus\Z/(81)\{B_6^{13}\}\\
&\oplus\Z/(81)\{B_7^{13},B_9^{13},B_{11}^{13},\ldots\}.
\end{split}
\end{align}
This computation will follow from the proof
of Proposition \ref{propcombo}(e).  
Example \ref{beta99} in Section \ref{Greek}
suggests that the algebraic beta family elements
$\beta^a_{9/9,1}$ and $\beta^a_{7/1,1}$ 
(the former being related to the 3-primary Kervaire
invariant problem) may be detected in the
submodule \eqref{specialcokernel}.
\end{rmk}

\section{The Adams-Novikov $E_2$-term for $Q(2)$}\label{ANSS}
The spectrum $Q(2)$ is the homotopy inverse 
limit of a semi-cosimplicial diagram of 
the form
\begin{equation}\label{q2bullet}
TMF\Rightarrow TMF\vee TMF_0(2)\Rrightarrow TMF_0(2),
\end{equation}
where $TMF$ and $TMF_0(2)$ are both
3-local
variants of the spectrum
of topological modular forms \cite{Hop:ATMF}.
The diagram (\ref{q2bullet}) can be viewed
as a more efficient version
of a tower of spectra used by the authors of
\cite{GHMR} in their study of the 3-primary
$K(2)$-local sphere.  
In this section, we describe the Adams-Novikov $E_2$-term
for $Q(2)$ in terms of the data
in (\ref{q2bullet}) and we set up the spectral sequence 
we will use to compute it.  
In particular, we will prove Propositions \ref{reduction}
and \ref{filtration}.

\subsection{Setup of the double complex spectral sequence for $C^{*,*}$}\label{setup}

Our starting point is the definition
of the elliptic curve Hopf algebroid
$(B,\Gamma)$ introduced in Section \ref{Intro}.
\begin{defn}
The graded $\ZLOC$-algebras $B$ and $\Gamma$ are
defined as follows:
\[
B=\ZLOC[q_2,q_4,\Delta^{-1}]/(\Delta=q_4^2(16q_2^2-64q_4))
\]
with $\deg(q_2)=2$ and $\deg(q_4)=4$ (hence $\deg(\Delta)=12$),
and 
\[
\Gamma=B[r]/(r^3+q_2r^2+q_4r)
\]
with $\deg(r)=2$.
\end{defn}
The groups $\Ext^*$ are encoded as the cohomology
groups of the cobar resolution $C^*(\Gamma)$ for $(B,\Gamma)$
(\cite{Rav:MU}, A1.2.11),
a cochain complex of the form
\[
B\xrightarrow{d}\Gamma\xrightarrow{d}\Gamma\otimes\Gamma
\xrightarrow{d}\Gamma\otimes\Gamma\otimes\Gamma\xrightarrow{d}
\cdots
\]
where the differentials $d$ are defined in terms of the
structure maps of $(B,\Gamma)$
(the coproduct, the right and left units, etc.).
Formulas for these structre maps are 
given in \cite{Bauer}.

The Hopf algebroid $(B,\Gamma)$ is connected
to $TMF$ 
via elliptic curves.
Any Hopf algebroid co-represents the objects and morphisms
of a groupoid; in the case of $(B,\Gamma)$, the
corresponding groupoid
is that of non-singular elliptic curves with Weierstrass 
equation
\begin{equation}\label{Weier}
y^2=4x(x^2+q_2x+q_4)
\end{equation}
and isomorphisms
$x\mapsto x+r$ that preserve this Weierstrass
form.  If $\mathcal M$ is the moduli stack
of such elliptic curves over $\ZLOC$,
the Goerss-Hopkins-Miller theorem
\cite{Goerss:Hopkins}
gives a sheaf $\mathcal O_{ell}$ of $E_{\infty}$ ring spectra
on $\mathcal M$, and $TMF$ is defined as the global sections
of this sheaf, i.e., $TMF=\mathcal O_{ell}(\mathcal M)$.
As a result, there is a spectral sequence
\[
E_2^{*,*}=\Ext^{*,*}=\Ext^{*,*}_{\Gamma}(B,B)\Rightarrow\pi_*TMF
\]
whose $E_2$-term is the cohomology of $(B,\Gamma)$.
This the Adams-Novikov spectral sequence
for $TMF$. 

To recover $TMF_0(2)$, consider the groupoid
whose objects are elliptic curves as in (\ref{Weier}) 
but with the additional datum of 
a $\Gamma_0(2)$ structure (i.e., a choice
of order 2 subgroup).  There are no nontrivial
structure-preserving
isomorphisms $x\mapsto x+r$ in this case, so the 
underlying Hopf algebroid is the
trivial Hopf algebroid $(B,B)$.  If
$\mathcal M_0(2)$ is the moduli stack of such
elliptic curves over $\ZLOC$, the Goerss-Hopkins-Miller
theorem once again gives a sheaf 
of $E_{\infty}$ ring spectra lying over it, 
and we obtain $TMF_0(2)$ by taking global sections.
%
The Adams-Novikov $E_2$-term for
$TMF_0(2)$ is therefore
$\Ext^*_B(B,B)=B$.
The spectral sequence collapes at $E_2$ and yields
\[
\pi_{2k}TMF_0(2)=B_k
\]
where $B_k$ denotes the elements of $B$
of degree $k$.
%

The following proposition gives the
Adams-Novikov spectral 
sequence converging to $\pi_*Q(2)$, whose 
$E_2$-term is
stitched together from the
Adams-Novikov $E_2$-terms for $TMF$ and $TMF_0(2)$
according to the maps in (\ref{q2bullet}).

\begin{prop}[\cite{Beh:Mod}, Section 1.4.2] 
\label{doubleprop}  
The Adams-Novikov $E_2$-term for $Q(2)$ is
the cohomology of the totalization of the double complex
$C^{*,*}$ given by
\begin{equation}\label{double}
C^*(\Gamma)\xrightarrow{\Phi}\overline C^*(\Gamma)\oplus B
\xrightarrow{\Psi}B\to0
\end{equation}
where $\overline C^*(\Gamma)$ is obtained
from $C^*(\Gamma)$ by multiplying its differentials
by $-1$, $B$ is viewed as a cochain complex concentrated
in Ext-degree 0,  
and the cochain complex maps $\Phi$, $\Psi$ are induced
by the corresponding maps of spectra in (\ref{q2bullet}).
\end{prop}

\subsection{Algebraic properties of $B$ and $\Ext^0$}
In this subsection we
lay the algebraic groundwork
for our computation by examining
the ring $B$ and the subring $\Ext^0\subset B$.
The latter is called the {\em ring of invariants} of the
Hopf algebroid $(B,\Gamma)$; it is the set of 
elements that are fixed by the right unit
structure map $\eta_R:B\to\Gamma$.  

Following \cite{Behrens:Rational},
we begin by defining a new
element $\mu\in B$:
\begin{equation*} 
\mu:=16q_2^2-64q_4,
\end{equation*}
an element of degree 4.
For computational convenience, 
we will replace $q_4$ and $\mu$
by scalar multiples of themselves,
namely
\begin{align}\label{COV}
\begin{split}
s&:=8q_4,\\
t&:=\mu/8,
\end{split}
\end{align}
thus $\deg(s)=\deg(t)=4$.  
[Note: While we also use $s$ and $t$
to refer to the bidegrees $(s,t)$ 
in $E_2^{*,*}Q(2)$, we believe their
meanings will always be clear from the
context.]
\begin{lem}\label{mu}
As a $\ZLOC$-algebra,
\[
B=\ZLOC[q_2,q_4,q_4^{-1},\mu^{-1}]/
(\mu=16q_2^2-64q_4)
\]
and thus
$\{s^it^jq_2^{\epsilon}:i,j\in\mathbb Z,
\epsilon=0\text{ or }1\}$
is a basis for $B$ as a $\ZLOC$-module.
\end{lem}
\begin{proof}
Since $\Delta=q_4^2\mu$,
inverting $\Delta$ is equivalent
to inverting $q_4$ and $\mu$, which
proves the first statement.  The second
statement follows from
(\ref{COV}) and the relation
$q_2^2=(\mu+64q_4)/16$.
\end{proof}
\begin{lem}
$B_0=\ZLOC[j_B,j_B^{-1}]$,
where 
\begin{equation}\label{JB}
j_B:=s/t.
\end{equation}
\end{lem}
\begin{proof}
The only elements $s^it^jq_2^{\epsilon}$
in $B_0$ are those with $i=-j$ and $\epsilon=0$.  
\end{proof}

\begin{defn}\label{submodules}
Given $i\leq j\in\Z$ and $\epsilon=0$ or 1,
define submodules
\[
V_{i,j,\epsilon}:=\Z_{(3)}\{s^it^jq_2^{\epsilon}, 
s^jt^iq_2^{\epsilon}\}\subset B
\]   
free of rank 1 if $i=j$, and free of rank 2 otherwise.
\end{defn}
\begin{lem}\label{BDECOMP}
As a $\ZLOC$-module,
${\displaystyle B=
\bigoplus_{i\leq j,\;\epsilon=0,1} V_{i,j,\epsilon}}$.
\end{lem}
\begin{proof}
This follows from Lemma \ref{mu}.
\end{proof}
We will see in Subsection \ref{hsubsection}
that the following elements
form a basis of eigenvectors for $B$ with respect
to the map $h:B\to B$ from Proposition \ref{filtration}.
\begin{defn}\label{evectors}
For $i<j\in\mathbb Z$, 
\[
a_{i,j}:=s^it^j-s^jt^i,\quad
\overline a_{i,j}:=s^it^j+s^jt^i,\quad
b_{i,j}:=a_{i,j}q_2,\quad
\overline b_{i,j}:=\overline a_{i,j}q_2
\]
and for $\epsilon=0$ or 1, 
$c_i^{\epsilon}:=s^it^iq_2^{\epsilon}$.
\end{defn}
The elements $\{a_{i,j}\}$ and $\{b_{i,j}\}$
from Definition \ref{evectors}
will be key in Section \ref{connecting}
when we compute the connecting homomorphisms in the 
long exact sequence from Proposition \ref{filtration}.
The following definition gives a convenient
enumeration of these elements for our
study of $\delta^1$ in Subsection \ref{deltaonethree}.
\begin{defn}\label{defnAB} For $0\leq v\in\Z$ and $m\in\Z$,
\[
A_v^m:=a_{\floor{\frac{m-1}2}-v,\ceil{\frac{m+1}2}+v},\quad
B_v^m:=b_{\floor{\frac{m-1}2}-v,\ceil{\frac{m+1}2}+v}.
\]
\end{defn}
Hereafter we denote the ring of invariants $\Ext^0$
by $MF$.  The following proposition
is an explicit description of $MF$ proven in
\cite{Del:Ell}.
\begin{prop}\label{MF}
If 
\begin{align}\label{ctoq}
\begin{split}
c_4&:=\mu+16q_4=2s+8t,\\
c_6&:=4q_2(8q_4-\mu)=4q_2(s-8t),\\
\end{split}
\end{align}
then
\[
MF=\ZLOC[c_4,c_6,\Delta,\Delta^{-1}]/
(1728\Delta=c_4^3-c_6^2)
\]
where $\Delta=q_4^2\mu=s^2t/8$ as before,
 $\deg(c_4)=4$, and
$\deg(c_6)=6$.
\end{prop}
\begin{rmk}
The notation ``$MF$'' stands for ``modular forms.''
Indeed, the ring $MF\otimes\mathbb C$
is the ring of modular forms over $\mathbb C$ for
the full modular group 
$\operatorname{SL}(2,\mathbb Z)$.
Note also that $B\otimes\mathbb C$
is the ring of modular forms over $\mathbb C$
for the congruence subgroup
$\Gamma_0(2)\subset\operatorname{SL}(2,\mathbb Z)$.
\end{rmk}
In the following definition
we identify bases for $MF$
and some of its $\ZLOC$-submodules
that will prove useful
for our computations in Sections
\ref{heart} and \ref{connecting}.
Note that $\mathcal B_{MF}^{\neq0}$
defined below appears in Proposition
\ref{gandh} (see Section \ref{Intro}).
\begin{defn}\label{MFBASIS}
Let
\begin{align*}
\mathcal B_{MF}&:=
\{c_4^nc_6^{\epsilon}\Delta^{\ell}:
n\geq0,\ell\in\mathbb Z,\epsilon=0
\text{ or }1\},\\
\mathcal B_{MF}^{\neq0}&:=
\{x\in\mathcal B_{MF}:\deg(x)\neq0\}\subset
\mathcal B_{MF}
\end{align*}
and, for any $m\in\Z$ and $\epsilon=0$ or 1, 
\[
\mathcal B^{\epsilon,m}_{MF}:=\{
c_4^nc_6^{\epsilon}\Delta^{\ell}:n+3\ell+\epsilon=m
\}\subset\mathcal B_{MF}.
\]
\end{defn}
\begin{lem}\label{basisMF}
The set $\mathcal B_{MF}$
is a basis for $MF$ as a $\ZLOC$-module.
\end{lem}
\begin{proof}
This follows from the relation
$c_6^2=c_4^3-1728\Delta$ in $MF$.
\end{proof}
\begin{lem}\label{MFZERO}
$MF_0=\ZLOC[j_{MF}]$,
where 
\begin{equation}\label{JMF}
j_{MF}:=c_4^3/\Delta
\end{equation}
is the $j$-invariant
(\cite{Slv}, Section III.1).
\end{lem}
\begin{proof}
First note that 
$c_4^3-c_6^2$ is irreducible in $MF$.  To
see this, we temporarily put $X=c_4$ and $Y=c_6$,
in which case it suffices to show that 
$Y^2-X^3=-(c_4^3-c_6^2)$ is irreducible.  
Suppose not.  Then we may write
\[
Y^2-X^3=(Y+f(X))(Y+g(X))
\]
where $f$ and $g$ are polynomials
in $X$, $f(X)g(X)=-X^3$ and $f(X)=-g(X)$.
In particular, $[g(X)]^2=X^3$, which
is impossible.

Since $c_4^3-c_6^2$ is irreducible,
the only basis elements $c_4^nc_6^{\epsilon}\Delta^{\ell}$
in $MF_0$ are those with $3\ell=-n$
and $\epsilon=0$.
\end{proof}
We now give notation for the submodules 
of $MF$ spanned by the sets 
$\mathcal B^{\epsilon,m}_{MF}$ in Definition
\ref{MFBASIS}.
\begin{defn}
Given $m\in\Z$ and $\epsilon=0$ or 1,
define submodules
\[
W_{\epsilon,m}:=\ZLOC\{
\mathcal B^{\epsilon,m}_{MF}\}\subset MF.
\]
\end{defn}
\begin{lem}As a $\ZLOC$-module,
\[
MF=W_{0,0}\oplus
\left(\bigoplus_{
\substack{
m\in\Z,\epsilon=0,1,\\(\epsilon,m)\neq(0,0)}}W_{\epsilon,m}
\right).
\]
\end{lem}
\begin{proof}
By degree counting, $MF_0=W_{0,0}$.
The result then follows from Lemmas \ref{basisMF} and \ref{MFZERO}
and the union decomposition
\[
\mathcal B_{MF}^{\neq0}=\bigcup_{\substack{
m\in\Z,\epsilon=0,1,\\(\epsilon,m)\neq(0,0)}}
\mathcal B_{MF}^{\epsilon,m}.
\]
\end{proof}
For an element $c_4^{n}\Delta^{\ell}\in\mathcal B_{MF}^{0,m}$, the largest possible value of $\ell$ is
\[
\LMAX:=\floor{\frac m3},
\]
while for $c_4^{n}c_6\Delta^{\ell}\in\mathcal B_{MF}^{1,m}$
it is
\[
\LMAXB:=\floor{\frac{m-1}3}.
\]
This allows us to give the following enumeration of the elements in $\mathcal B_{MF}^{\neq0}$, convenient for our
study of $\delta^1$ in Subsection \ref{deltaonethree}.
\begin{defn}\label{defnCD}
For $0\leq v\in\Z$ and $m\in\Z$,
\[
C_v^m:=c_4^{m-3\LMAX+3v}\Delta^{\LMAX-v},\quad
D_v^m:=c_4^{m-3\LMAXB-1+3v}c_6\Delta^{\LMAXB-v}
\]
so that 
$\mathcal B_{MF}^{0,m}=\{C_0^m,C_1^m,C_2^m,\ldots\}$
and 
$\mathcal B_{MF}^{1,m}=\{D_0^m,D_1^m,D_2^m,\ldots\}$.
\end{defn}
\begin{rmk}\label{tdegrees}
The enumerations in Definitions \ref{defnAB}
and \ref{defnCD} are analogous in terms of how 
the integer $m$ compares with the polynomial degree.  Specifically,
\begin{equation}
\deg(A_v^m)=\deg(C_v^m)=4m,\quad
\deg(B_v^m)=\deg(D_v^m)=4m+2.
\end{equation}
\end{rmk}

\subsection{Maps of the double complex}
\label{maps}
In this subsection 
we describe four Hopf
algebroid maps, denoted $\psi_d$, $\phi_f$, $\phi_q$, and 
$\psi_{[2]}$, that assemble to 
give $\Phi$ and $\Psi$ 
as follows:
\begin{align}\label{PhiPsi}
\begin{split}
\Phi&=(\psi_{[2]}\oplus\phi_q)
-(1_{\Gamma}\oplus\phi_f),\\
\Psi&=\psi_d-\phi_f+1_B.
\end{split}
\end{align}
This yields the diagram
\[
\Gamma\xrightarrow{\Phi}
\Gamma\oplus B\xrightarrow{\Psi} B\to0
\]
of $\ZLOC$-modules inducing
the double cochain complex
(\ref{double}) in Proposition \ref{doubleprop}.

Each of $\psi_d$, $\phi_f$, $\phi_q$, and 
$\psi_{[2]}$
corresponds to a maneuver with elliptic curves
(see Remark \ref{stackremark} below) and is defined
by the effect of the maneuver on Weierstrass equations,
as computed in Section 1.5 of \cite{Beh:Mod}
(where they are denoted $\psi_d^*$, 
$\phi_f^*$, $\phi_q^*$, and 
$\psi_{[2]}^*$, respectively).
We briefly summarize those computations
here.
Since each map is a Hopf algebroid morphism, those with source $(B,B)$ are determined by their
values on $q_2$ and $q_4$, while those with source $(B,\Gamma)$ are determined by their values on
$q_2$, $q_4$, and $r$.

Given an elliptic curve $C$ over $\ZLOC$ with Weierstrass equation as in (\ref{Weier}) and an 
order 2 subgroup $H$,
$\psi_d:(B,B)\to(B,B)$ records the effect 
on $q_2$ and $q_4$ when 
$C$ is replaced by its quotient $C/H$, or equivalently,
when the degree 2 isogeny $C\to C/H$
is replaced by its dual isogeny $C/H\to C$.
The effect is
\begin{align*}
\psi_d:q_2&\mapsto-2q_2,\\
q_4&\mapsto q_2^2-4q_4.
\end{align*}

If $C$ is an elliptic curve as before,
then $\phi_f:(B,\Gamma)\to(B,B)$ forgets the choice of order
2 subgroup $H\subset C$.  Simply forgetting this
extra structure does not impact the coefficients
$q_2$ and $q_4$ but it does impact which
elliptic curve morphisms are allowed.  
Since there are no transformations $x\mapsto x+r$
that preserve $H$, $\phi_f$ is given by
\begin{align*}
\phi_f:q_2&\mapsto q_2,\\
q_4&\mapsto q_4,\\
r&\mapsto0.
\end{align*}

For $\phi_q:(B,\Gamma)\to(B,B)$, the relation
$\phi_q=\psi_d\circ \phi_f$
imposed by the semi-cosimplicial 
structure of (\ref{q2bullet}) implies
\begin{align*}
\phi_q:q_2&\mapsto-2q_2,\\
q_4&\mapsto q_2^2-4q_4,\\
r&\mapsto 0.
\end{align*} 


The map $\psi_{[2]}$
can be viewed either as a self-map of $(B,\Gamma)$
or as a self-map of $(B,B)$ (\cite{Beh:Mod}, Section 1.1).  In either case,
$\psi_{[2]}$ corresponds to taking the
quotient of $C$ by its subgroup $C[2]$
of points of order 2.  The standard elliptic curve
addition formulas show that, on the level of Weierstrass
equations, this corresponds to replacing $q_2$
by $2^2q_2$ and $q_4$ by $2^4q_4$.  Moreover,
the allowable transformations in this case
are of the form $x\mapsto x+2^2r$.  Thus,
as a self-map of $(B,\Gamma)$,
\begin{align*}
\psi_{[2]}:q_2&\mapsto4q_2,\\
q_4&\mapsto16q_4,\\
r&\mapsto4r
\end{align*}
and restriction yields the corresponding
self-map of $(B,B)$.

Combined with (\ref{PhiPsi}) the above formulas yield
\begin{align}\begin{split}\label{phiformula}
\Phi:q_2&\mapsto(3q_2,-3q_2),\\
q_4&\mapsto(15q_4,q_2^2-5q_4),\\
r&\mapsto(3r,0)
\end{split}
\end{align}
and $\Psi:(x,y)\mapsto\psi_d(y)-\phi_f(x)+y$ 
for $(x,y)\in \Gamma\oplus B$.
\begin{rmk}\label{stackremark}
The semi-cosimplicial diagram (\ref{q2bullet}) 
underlying $Q(2)$ is the topological realization of a 
semi-simplicial diagram of stacks
\begin{equation}\label{stacks}
\mathcal M\Leftarrow
\mathcal M\coprod\mathcal M_0(2)\Lleftarrow
\mathcal M_0(2).
\end{equation}
The stacks $\mathcal M$ and $\mathcal M_0(2)$ 
are categories 
fibered in groupoids over the category
of $\ZLOC$-affine schemes.  
Given a $\ZLOC$-algebra $T$, 
the groupoid lying over $\operatorname{Spec}(T)$
in $\mathcal M$
is the one co-represented by $(B,\Gamma)$, while
the groupoid lying over $\operatorname{Spec}(T)$
in $\mathcal M_0(2)$ is the one co-represented
by $(B,B)$.  The Hopf algebroid maps defined 
in this subsection correspond to the morphisms
of stacks in (\ref{stacks}), and the elliptic curve 
maneuvers can be interpreted as
descriptions of what these stack morphisms
do on the level of $T$-points.
\end{rmk}

\subsection{Proof of Proposition \ref{reduction}}\label{dcss}

Recall from Section \ref{setup} 
that the Adams-Novikov $E_2$-term for $Q(2)$ is the target of the double complex spectral sequence for $C^{*,*}$, which is
a first quadrant double complex of the form
\[
\xymatrix{
C^*(\Gamma)\ar[r]^-{\Phi}\ar@{=}[d] & 
\overline C^*(\Gamma)\oplus B
 \ar[r]^-{\Psi}\ar@{=}[d] & B\ar[r]\ar@{=}[d] & 0\ar@{=}[d] \\
C^{0,*} & C^{1,*} & C^{2,*} & C^{3,*}\,
}
\]
The vertical differentials are induced by the cobar
complex differentials for $(B,\Gamma)$
and are formally the
$d_0$-differentials of the double complex spectral sequence;  
the horizontal maps 
$\Phi$ and $\Psi$ were defined in the previous section
and will induce the $d_1$-differentials.

Expanding the three nontrivial columns of $C^{*,*}$
gives
\begin{equation}\label{EZERO}
\xymatrix{
\stackrel{\vdots}{\Gamma\otimes\Gamma}\ar[r]^-{\Phi} & \stackrel{\vdots}{\Gamma\otimes\Gamma}\ar[r] & 
\stackrel{\vdots}{0} \\
\Gamma\ar[r]^-{\Phi}\ar[u]^{d} & \Gamma\ar[r]\ar[u]^{-d} & 0\ar[u] \\
B\ar[r]^-{\Phi}\ar[u]^{d} & B\oplus B\ar[r]^-{\Psi}\ar[u]^{-d\oplus 0} & B\ar[u] 
}
\end{equation}
and turning to the $E_1$-page yields 
the following result.
\begin{lem}\label{EONELEM}
Taking cohomology with respect to
the vertical differentials in (\ref{EZERO})
gives
\begin{equation}\label{EONE}
\xymatrix{
\stackrel{\vdots}{\Ext^2}\ar[r]^-{0} & \stackrel{\vdots}{\Ext^2}\ar[r] & 
\stackrel{\vdots}{0} \\
\Ext^1\ar[r]^-{0}\ar[u] & \Ext^1\ar[r]\ar[u] & 0\ar[u] \\
MF\ar[r]^-{\Phi}\ar[u] & MF\oplus B \ar[r]^-{\Psi}\ar[u] & B\ar[u] 
}
\end{equation}
\end{lem}
\begin{proof}
We know $\Ext^0=MF$ by Proposition \ref{MF}.
The computation of the homotopy groups
of $TMF$ by Hopkins and Miller \cite{Bauer} 
shows that
$\Ext^n$ is entirely 3-torsion for $n\geq1$.  
Equation \eqref{phiformula}
therefore implies 
$\Phi:\Ext^n\to\Ext^n$ must be identically
zero for $n\geq1$. 
\end{proof}
Lemma \ref{EONELEM} shows that 
the double complex spectral sequence
for $C^{*,*}$ has only two potentially
nontrivial differentials on its $E_1$-page:
$\Phi$ and $\Psi$.  By sparseness, 
the only
potentially nontrivial differential
on the $E_2$-page is a map
$\Ext^1\to\coker\Psi$
(denoted $\widetilde d$ in Proposition
\ref{reduction})
and $E_3=E_{\infty}$.  Therefore,
to obtain
the $E_{\infty}$-page from (\ref{EONE})
we need only replace the 0th row by
\[
H^0C^*\to H^1C^*\to\coker\widetilde d
\]
and replace $\Ext^1$ in the 0th column
by $\ker\widetilde d$.  This completes
the proof of Proposition \ref{reduction}.

\subsection{Proof of Proposition \ref{filtration}}\label{filt}

In this subsection we 
define the two-stage filtration of $C^*$
we shall use to compute $H^*C^*$ and we
prove Proposition \ref{filtration}.
For ease of notation, we will henceforth
denote by 1 the maps $1_{B}$,
$1_{\Gamma}$, and any maps they
induce; the meaning should
be clear from the context.
  
If $F^0=C^*$, $F^1=(MF\xrightarrow{\psi_{[2]}-1}MF\to0)$,
and $F^2$ is the trivial complex, 
then $F^0\supset F^1\supset F^2$ is our
filtration.  It induces a short exact 
sequence 
\begin{equation}\label{sesthree}
0\to C'\to C^*\to C''\to0
\end{equation}
of chain complexes, given by
\[
\xymatrix{
C': & 0\ar[r]\ar[d] 
    & B_{\,}\ar[r]^{\psi_d+1}\ar@{^{(}->}[d] 
    & B\ar@{=}[d] \\
C^*:  & MF\ar[r]^-{\Phi}\ar@{=}[d] 
    & B\oplus MF\ar[r]^-{\Psi}\ar@{>>}[d] 
    & B\ar[d]\\
C'':& MF\ar[r]^{\psi_{[2]}-1} & MF\ar[r] & 0
}
\]
\begin{defn}
$g:=\psi_{[2]}-1:MF\to MF$,\quad
$h:=\psi_d+1:B\to B$.
\end{defn}
Proposition \ref{filtration} 
follows from standard homological algebra
(see, e.g., Section 1.3 of \cite{Weibel}).
The map $\delta^0$ 
is the restriction of 
$\phi_q-\phi_f$ to $\ker g$, while the map
$\delta^1$ is the map induced by $-\phi_f$ on $\coker g$.

\section{Computation of the maps $g$ and $h$}\label{heart}

In this section we initiate  
our computation of the Adams-Novikov 
$E_2$-term for $Q(2)$ by computing the 
kernel and cokernel of the maps
$g:MF\to MF$ and $h:B\to B$ defined in
Section \ref{ANSS}.   
\subsection{A 3-divisibility result}
The following 
result in 3-adic analysis is one
we shall leverage numerous times
throughout the remainder of this paper.
\begin{lem}\label{adic}
\begin{enumerate}[(a)]
\item If $n$ is a nonzero even integer, then 
\[
\nu_3(4^n-1)=\nu_3(n)+1.
\]
\item If $n$ is an odd integer, then
\[
\nu_3(2^n+1)=\nu_3(n)+1.
\]
\end{enumerate}
\end{lem}
\begin{proof}
Let $|\cdot|$ denote
3-adic absolute value.
Fix an even integer $n>1$ (the case $n<-1$
will follow immediately), and let
\[
f(x)=(1+x)^n-1. 
\]
Recall that the (3-adic) functions $e^x$ and $\log(1+x)$ 
converge for 
$|x|\leq|3|$. 
Moreover,
$|\log(1+x)|=|x|$, $|e^x|=1$, and $|1-e^x|=|x|$ 
for any $|x|\leq|3|$.  But since
\[ 
f(x)=(1+x)^n-1=e^{n\log(1+x)}-1
\]
this implies that for $|x|\leq|3|$, 
\[
|f(x)|=|e^{n\log(1+x)}-1|=|n\log(1+x)|
=|n||\log(1+x)|=|n||x|.
\]
In particular, setting $x=3$  
yields $|f(3)|=|4^n-1|=|n||3|$,
which proves (a).

To prove (b), we need only
slightly alter the above argument.  
Fix an odd integer $n>0$.
Replacing $x$ by $-x$ in the definition of $f(x)$
yields a new function
\[
g(x)=(1-x)^n-1=e^{n\log(1-x)}-1
\]
and a similar analysis shows that if $|x|\leq|3|$,
then 
$|g(x)|=|n||x|$.
Setting $x=3$ as before yields
$|g(3)|=|(-2)^n-1|=|n||3|$.
But since $n$ is odd, this implies 
$|2^n+1| = |n||3|$. 
\end{proof}

\subsection{Kernel and cokernel of $g:MF\to MF$}\label{gsubsection}
If $x\in MF$, 
the formulas for $\psi_{[2]}$ in
Subsection \ref{maps} imply
\begin{equation}\label{gee}
g(x)=(2^{\deg(x)}-1)x.
\end{equation}
Since $2^{\deg(x)}-1=0$ if and only if $\deg(x)=0$,
Lemma \ref{MFZERO} implies
\[
\ker g=MF_0=\ZLOC[j_{MF}]=\bigoplus_{n\in\N}\ZLOC.
\]

Now suppose $x\in\mathcal B_{MF}^{\neq0}$.
The degree of $x$ must be even, say $\deg(x)=2k$.
By \eqref{gee}, $g(x)=(2^{\deg(x)}-1)x$, and 
Lemma \ref{adic}(a) implies
\[
\nu_3(2^{\deg(x)}-1)=\nu_3(4^{k}-1)=\nu_3(k)+1 
=\nu_3(\deg(x))+1.
\]
Thus
\[
\operatorname{im}g=\bigoplus_{x\in\mathcal B^{\neq0}_{MF}}
3^{\nu_3(\deg(x))+1}\ZLOC
\]
and the result for $\coker g$ in 
Proposition \ref{gandh} follows.

\subsection{Kernel and cokernel of $h:B\to B$}
\label{hsubsection}
We begin by studying $h$ on  
the submodules $V_{i,j,\epsilon}\subset B$
from Definition \ref{submodules}.
\begin{prop}\label{invariance}
Each $V_{i,j,\epsilon}$ is invariant 
under $h$, and $h\big|_{V_{i,j,\epsilon}}$
has a matrix representation with respect to
$\{s^it^jq_2^{\epsilon},s^jt^iq_2^{\epsilon}\}$ depending on
$i,j,\epsilon$ as follows:
\begin{enumerate}[(a)]
\item If $i<j$ and $\epsilon=0$,
\[
h\big|_{V_{i,j,\epsilon}}=
\left[
\begin{array}{cc}  
1 & 4^{i+j}\\
4^{i+j} & 1
\end{array}\right]
\]
with eigenvectors $\overline a_{i,j}$, $a_{i,j}$
(see Def.\ \ref{evectors}) and corresponding
eigenvalues
\[
\overline{\lambda}_{i,j}:=1+4^{i+j}\in\ZLOC^{\times},\quad  
\lambda_{i,j} :=1-4^{i+j}\notin\ZLOC^{\times}.
\]
\item If $i<j$ and $\epsilon=1$,
\[
h\big|_{V_{i,j,\epsilon}}=
\left[
\begin{array}{cc}  
1 & -2^{2i+2j+1}\\
-2^{2i+2j+1} & 1
\end{array}\right]
\]
with eigenvectors $b_{i,j}$, $\overline b_{i,j}$
(see Def.\ \ref{evectors}) and corresponding
eigenvalues
\[
\rho_{i,j}:=1+2^{2i+2j+1}\notin\ZLOC^{\times},\quad
\overline{\rho}_{i,j}:=1-2^{2i+2j+1}\in\ZLOC^{\times}.
\]
\item $h\big|_{V_{i,i,\epsilon}}$
is multiplication by 
$16^i(-2)^{\epsilon}+1\in\ZLOC^{\times}$.
\end{enumerate}
\end{prop}
\begin{proof}
The formulas for $\psi_d$
in Section \ref{maps} imply that
\begin{equation}\label{aitch}
h(s^it^jq_2^{\epsilon})=
4^{i+j}(-2)^{\epsilon}s^jt^iq_2^{\epsilon}
+s^it^jq_2^{\epsilon}
\end{equation}
which proves invariance and gives the
matrix for $h\big|_{V_{i,j,\epsilon}}$ in all three
cases (the matrix in case (c) being $1\times1$).
The eigenvectors and eigenvalues can be found by
a direct computation.  The eigenvalues
$\overline\lambda_{i,j}$, $\overline\rho_{i,j}$,
and $16^i(-2)^{\epsilon}+1$
are congruent to 2 modulo 3 and therefore are
invertible in $\ZLOC$.
Lemma \ref{adic}(a)
implies $\lambda_{i,j}$ is 3-divisible, and
Lemma \ref{adic}(b) implies $\rho_{i,j}$ is
3-divisible.
\end{proof}
Lemma \ref{BDECOMP} and Proposition \ref{invariance} show
that the set
\[
\{a_{i,j}; \overline a_{i,j};
b_{i,j}; \overline b_{i,j}; c_i^{\epsilon}
:i<j\in\mathbb Z,\epsilon=0,1\}
\]
is a $\ZLOC$-basis of eigenvectors for $B$
relative to $h$.  The generators
$\overline a_{i,j}$, $\overline b_{i,j}$,
and $c_i^{\epsilon}$ all map to unit multiples
of themselves under $h$ by Proposition 
\ref{invariance}, and hence are not 
contained in the kernel.  Since
$\rho_{i,j}\neq0$ for all $i<j\in\mathbb Z$,
the generators $b_{i,j}$ are also not in the kernel.
Finally, since $\lambda_{i,j}=0$ if and only if
$i=-j$, the only generators of the form
$a_{i,j}$ that lie in the kernel of $h$ are
$\{a_{-i,i}:i\geq1\}$.  Thus
\[
\ker h=\ZLOC\{a_{-i,i}:i\geq1\}=\bigoplus_{n\in\N}\ZLOC.
\]

Lemma \ref{adic}(a) implies
\[
\nu_3(\lambda_{i,j})=\nu_3(1-4^{i+j})
=\nu_3(i+j)+1
\]
and similarly, Lemma \ref{adic}(b)
implies
\[
\nu_3(\rho_{i,j})=\nu_3(1+2^{2i+2j+1})
=\nu_3(2i+2j+1)+1.
\]
These results, together with Proposition 
\ref{invariance}, imply that
\[
\operatorname{im} h=
\left(\bigoplus_{x\in\{\overline a_{i,j};
\overline b_{i,j}; c_i^{\epsilon}\}}\ZLOC\right)\oplus
\left(
\bigoplus_{i<j\in\mathbb Z}
\left(3^{\nu_3(i+j)+1}\ZLOC\oplus
3^{\nu_3(2i+2j+1)+1}\ZLOC\right)
\right).
\]
The result for $\coker h$ in
Proposition \ref{gandh} follows
if we take $m=i+j$.  This completes
the proof of Proposition \ref{gandh}.

\section{Computation of the connecting
homomorphisms $\delta^0$ and $\delta^1$}
\label{connecting}

In this section we compute the kernel
and cokernel of $\delta^0$ and $\delta^1$.
Using these computations, we prove Theorem
\ref{cohomology}.

\subsection{The kernel and cokernel of
$\delta^0:\ker g\to\ker h$}\label{deltazerothree}

From the results of Section \ref{heart},
\[
\delta^0=\phi_q-\phi_f:\Z_{(3)}[j_{MF}]\to\Z_{(3)}
\{a_{-i,i}:i\geq1\}.
\] 
\begin{prop}\label{delzero}
$\ker\delta^0=\Z_{(3)}\{1_{MF}\}$
and $\coker\delta^0=\ZLOC\{a_{-i,i}:i\geq1, \rm{\, odd}\}$.
\end{prop}
\begin{proof}
The map $\delta^0$ is completely determined
by where it sends nonnegative powers of $j_{MF}$.
The formula for $\phi_q$ in
Section \ref{maps} implies
\begin{align}\label{ONMU}
\begin{split}
\phi_q:q_4&\mapsto\mu/16,\\
\mu&\mapsto 256q_4.
\end{split}
\end{align}
Combining (\ref{ONMU}) with
the formula for $\phi_f$,
the formulas (\ref{COV}) for $s$ and $t$,
and Definition \ref{evectors}, yields
\begin{align}\label{delzerojmf}
\begin{split}
\delta^0(j_{MF}^k)&=(\phi_q-\phi_f)(c_4^{3k}\Delta^{-k})\\
&=(256q_4+\mu)^{3k}\mu^{-2k}q_4^{-k}
-(\mu+16q_4)^{3k}q_4^{-2k}\mu^{-k}\\
&=\frac{2^{6k}(4s+t)^{3k}}{t^{2k}s^k}
-\frac{2^{6k}(4t+s)^{3k}}{s^{2k}t^k}\\
&=\sum_{r=0}^{3k}\binom{3k}r 2^{12k-2r}
(s^{2k-r}t^{r-2k}-s^{r-2k}t^{2k-r})\\
&=\sum_{r=0}^{3k}\binom{3k}r 2^{12k-2r}a_{2k-r,r-2k}\\
&=2^{8k}
\sum_{v=1}^k\left(
\binom{3k}{2k+v}4^{-v}-\binom{3k}{2k-v}4^v
\right)a_{-v,v}\\
&\hspace{.4truein}-2^{8k}\sum_{v=k+1}^{2k}\binom{3k}{2k-v}4^v a_{-v,v}.
\end{split}
\end{align}
Thus, with respect to
$\{1, j_{MF}, j_{MF}^2,\ldots\}$ and $\{a_{-1,1}, a_{-2,2},\ldots\}$, 
\begin{equation}\label{deltazeromatrix}
\delta^0=\left[
\begin{array}{ccccc}
0 & * & * & * & \, \\
\vdots & u_1 & * & * & \,  \\
\, & 0 & * & * & \, \\
\, & \vdots & u_2 & * & \,  \\
\, & \, & 0 & * & \, \\
\, & \, & \vdots & u_3 & \,  \\
\, & \, & \, & 0 & \ddots 
\end{array}\right]
\end{equation}
where $u_k=-2^{12k}\in\ZLOC^{\times}$ for $k\geq1$.
\end{proof}

\subsection{The connecting map 
$\delta^1:\coker g\to\coker h$}\label{deltaonethree}

From the results of Section \ref{heart}, 
\begin{align*}
\delta^1=-\phi_f:
\left(\bigoplus_{x\in\mathcal B^{\neq0}_{MF}}\Z/(3^{\nu_3(\deg(x))+1})\right)
&\oplus\ZLOC[j_{MF}]\\
\to &\bigoplus_{i<j\in\mathbb Z}
\left(\mathbb Z/(3^{\mu_3(3i+3j)})\oplus
\mathbb Z/(3^{\mu_3(6i+6j+3)})
\right)
\end{align*}
where $\Z/(3^{\mu_3(3i+3j)})$ 
is generated by $a_{i,j}$ and 
$\Z/(3^{\mu_3(6i+6j+3)})$
is generated by $b_{i,j}$.
In particular, $\delta^1$  
is completely determined by where it sends
nonnegative powers of $j_{MF}$ and
the elements of $\mathcal B^{\neq0}_{MF}$.
\begin{lem}\label{AIJ}
For all $i<j\in\Z$,
$s^it^j=a_{i,j}/2$
and $s^it^jq_2=b_{i,j}/2$ in $\coker h$.
\end{lem}
\begin{proof}
From Definition \ref{evectors},
\[
s^it^j=\frac{a_{i,j}+\overline a_{i,j}}2,
\quad
s^it^jq_2=\frac{b_{i,j}+\overline b_{i,j}}2
\]
and $\overline a_{i,j}=\overline b_{i,j}=0$ 
in $\coker h$.
\end{proof}
\begin{prop}\label{delone}
$\ker\left(
\delta^1\big|_{\ZLOC[j_{MF}]}
\right)=\ZLOC\{1_{MF}\}$ and
$\coker\left(
\delta^1\big|_{\ZLOC[j_{MF}]}
\right)=\ZLOC\{a_{-i,i}:i\geq1, \rm{\, odd}\}$.
\end{prop}
\begin{proof}
By Lemma \ref{AIJ} and
the formula for $\phi_f$
from Section \ref{maps}, a computation
similar to \eqref{delzerojmf} in
Proposition \ref{delzero} gives
\begin{align}\label{delonejmf}
\begin{split}
\delta^1(j_{MF}^k)&=
2^{8k-1}\sum_{v=1}^k\left(
\binom{3k}{2k+v}4^{-v}-\binom{3k}{2k-v}4^v
\right)a_{-v,v}\\
&\hspace{.4truein}-2^{8k-1}\sum_{v=k+1}^{2k}
\binom{3k}{2k-v}4^va_{-v,v}\\
&=\frac12\delta^0(j_{MF}^k)
\end{split}
\end{align}
so $\ker\left(
\delta^1\big|_{\ZLOC[j_{MF}]}
\right)=\ker\delta^0$ and 
$\coker\left(
\delta^1\big|_{\ZLOC[j_{MF}]}
\right)=\coker\delta^0$.
\end{proof}
We now study $\delta^1\big|_{W_{\epsilon,m}}$
for $m\in\Z$ and $\epsilon=0$ or 1
by finding matrix representations, as we did
with $\delta^0$ in \eqref{deltazeromatrix}.  
The set $\mathcal B_{MF}^{\epsilon,m}$
is an ordered basis for the source. Degree counting
shows that an ordered basis for the target is given by
$\{A_0^m,A_1^m,A_2^m,\ldots\}$
if $\epsilon=0$, and $\{B_0^m,B_1^m,B_2^m,\ldots\}$
if $\epsilon=1$.

By Lemma \ref{AIJ} and the formula
for $\phi_f$,
\begin{align}\label{epsiloniszero}
\begin{split}
\delta^1(c_4^n\Delta^{\ell})
&=-\phi_f(c_4^n\Delta^{\ell})
=-(\mu+16q_4)^n(q_4^2\mu)^{\ell}\\
&=-\frac{s^{2\ell}t^{\ell}(2s+8t)^n}{8^{\ell}}
=-8^{n-\ell}\sum_{r=0}^n\binom nr 4^{-r}s^{2\ell+r}t^{n+\ell-r}\\
&=-2^{3n-3\ell-1}\sum_{r=0}^n
\binom nr 4^{-r}
a_{2\ell+r,n+\ell-r}
\end{split}
\end{align}
and
\begin{align}\label{epsilonisone}
\begin{split}
\delta^1(c_4^nc_6\Delta^{\ell})
&=-\phi_f(c_4^nc_6\Delta^{\ell})
=-(\mu+16q_4)^n(q_4^2\mu)^{\ell}(4q_2(8q_4-\mu))\\
&=-\frac{q_2s^{2\ell}t^{\ell}(2s+8t)^n
(s-8t)}{2^{3\ell-2}}\\
&=
-4^{2n-2\ell+1}\sum_{r=0}^n\binom nr 4^{-r}
q_2\left(
s^{2\ell+r+1}t^{n+\ell-r}
-8s^{2\ell+r}t^{n+\ell-r+1}
\right)\\
&=-2^{4n-4\ell+1}
\sum_{r=0}^n \binom nr 4^{-r}\left(
b_{2\ell+r+1,n+\ell-r}
-8b_{2\ell+r,n+\ell-r+1}
\right).
\end{split}
\end{align}
\begin{rmk}\label{SORT}
To obtain matrix representations of 
$\delta^1\big|_{W_{\epsilon,m}}$, the
right-hand sums of
(\ref{epsiloniszero}) (resp.\ (\ref{epsilonisone}))
must be put
solely in terms of the generators 
$a_{i,j}$ (resp.\ $b_{i,j}$) with $i<j$,
because of the identities
\begin{equation}\label{swap}
a_{i,j}=-a_{j,i},\quad
b_{i,j}=-b_{j,i}.
\end{equation}
Note that this was done implicitly in the proofs of 
Propositions \ref{delzero} and \ref{delone}.
\end{rmk}
\begin{prop}\label{propcombo}
\begin{enumerate}[(a)]
\item $\ker\left(\delta^1\big|_{W_{0,m}}\right)=
\begin{cases}
0, & m<0,\\
\Z/(3^{\nu_3(m)+1}), & m>0.
\end{cases}$
\item For all $0\neq m\in\Z$,
${\displaystyle
\coker\left(\delta^1\big|_{W_{0,m}}\right)
=\bigoplus_{n\in\N}\Z/(3^{\nu_3(m)+1})}$.
\item For $m\leq0$, 
$\ker\left(\delta^1\big|_{W_{1,m}}\right)=0$ and
${\displaystyle
\coker\left(\delta^1\big|_{W_{1,m}}\right)
=\bigoplus_{n\in\N}\Z/(3^{\nu_3(2m+1)+1})}$.
\item For $m>0$ and $m\nequiv13\bmod27$,
$\ker\left(\delta^1\big|_{W_{1,m}}\right)
=\Z/(3^{\nu_3(2m+1)+1})$ and
\[
\coker\left(\delta^1\big|_{W_{1,m}}\right)
=\bigoplus_{n\in\N}\Z/(3^{\nu_3(2m+1)+1}).
\]
\item For $m>0$ and $m\equiv13\bmod27$,
$\coker\left(\delta^1\big|_{W_{1,m}}\right)$
has, as a direct summand, ${\displaystyle\bigoplus_{n\in\N}\Z/(3^{\nu_3(2m+1)+1})}$.
\end{enumerate}
\end{prop}
We now establish the following convenient
notational conventions.
\begin{defn}\label{paradigm}
\begin{enumerate}[(a)]
\item
If $c_4^n\Delta^{\ell}\in\BMF^{0,m}$
(resp.\ $c_4^nc_6\Delta^{\ell}\in\BMF^{1,m}$),
the $a_{i,m-i}$ term of
$\delta^1(c_4^n\Delta^{\ell})$
(resp.\ the $b_{i,m-i}$ term of $\delta^1(c_4^nc_6\Delta^{\ell})$)
with the {\em least} first subscript $i$
will be denoted the {\em leading term}, 
and the remaining
terms will be denoted {\em higher order terms}.
\item The symbol $\doteq$ will denote equality
up to multiplication by a unit in $\ZLOC$.
\item If $M$ is a matrix with columns
$M_1,\ldots,M_v$ and $N$ is a matrix
with (possibly infinitely many) columns
$N_1,N_2,\ldots$, then $M\boxplus N$
will denote the matrix with columns
$M_1,\ldots,M_v,N_1,N_2,\ldots$.
\end{enumerate}
\end{defn}
To prove Proposition \ref{propcombo}, we
will need the following four lemmas.
\begin{lem}\label{nocsix}
If $\ell\neq0$,
\[
\delta^1(c_4^n\Delta^{\ell})\doteq
\begin{cases}
a_{\ell,n+2\ell}+\HOT, & \ell>0 \\
a_{2\ell,n+\ell}+\HOT, & \ell<0.
\end{cases}
\]
\end{lem}

\begin{proof}
We use (\ref{epsiloniszero}) and (\ref{swap})
as in the proof of the previous lemma.  
For $-\ell\geq n$, 
\[
\delta^1(c_4^n\Delta^{\ell})=
-2^{3n-\ell-1}\sum_{r=2\ell}^{2\ell+n}\binom n{r-2\ell}
4^{-r}a_{r,n+3\ell-r}
\]
with leading term $-2^{3n-5\ell-1}a_{2\ell,n+\ell}$.  
For $\ell\geq n$,
\[
\delta^1(c_4^n\Delta^{\ell})=
2^{n-5\ell-1}\sum_{r=\ell}^{\ell+n}\binom n{n+\ell-r}4^{r}
a_{r,n+3\ell-r}
\]
with leading term $2^{n-3\ell-1}a_{\ell,n+2\ell}$.
For $|\ell|<n$,
\begin{align*}
\delta^1(c_4^n\Delta^{\ell})&=
2^{n-5\ell-1}\sum_{r=\ell}^{\left\lfloor\frac{n+3\ell-1}2\right\rfloor}\binom n{n+\ell-r}
4^{r}a_{r,n+3\ell-r}\\
&-2^{3n+\ell-1}\sum_{r=2\ell}^{\left\lfloor\frac{n+3\ell-1}2\right\rfloor}
\binom n{r-2\ell}4^{-r} a_{r,n+3\ell-r}
\end{align*}
with leading term $2^{n-3\ell-1}a_{\ell,n+2\ell}$ if $\ell>0$
and $-2^{3n-3\ell-1}a_{2\ell,n+\ell}$ if $\ell<0$.
\end{proof}

\begin{lem}\label{elliszero} When $\ell=0$,
\begin{align*}
\delta^1(c_4^nc_6)&\doteq
4\sum_{r=1}^{\left\lfloor\frac n2\right\rfloor}\binom n{r-1}4^{-r}b_{r,n+1-r}
-4^{-n}\sum_{r=0}^{\left\lfloor\frac n2\right\rfloor}
\binom nr 4^r b_{r,n+1-r}\\
&\hspace{.4truein}-8\sum_{r=0}^{\left\lfloor\frac n2\right\rfloor}\binom nr 4^{-r} b_{r,n+1-r}
+2^{-2n+1}\sum_{r=1}^{\left\lfloor\frac n2\right\rfloor}\binom n{r-1}4^r b_{r,n+1-r}
\end{align*}
and the leading term $(-4^{-n}-8)b_{0,n+1}$ is zero
in $\coker h$.
\end{lem}
\begin{proof}
The formula is obtained from 
(\ref{epsilonisone})
by setting $\ell=0$.
Each $b_{r,n+1-r}$ generates a copy
of $\Z/(3^{\nu_3(2n+3)+1})$ in 
$\coker h$, and since Lemma \ref{adic}(b)
implies
\[
\nu_3(-4^{-n}-8)=\nu_3(2^{-2n-3}+1)
=\nu_3(-2n-3)+1,
\]
the leading term does indeed vanish.
\end{proof}
\begin{lem} \label{lemcombo}
\begin{enumerate}[(a)]
\item 
$\delta^1(c_4^n)=0$ for $n\geq1$.
\item $\delta^1(c_4^nc_6)=0$ except when 
$n+1\equiv 13\bmod 27$.  In these exceptional
cases, the $b_{1,n}$ term is always nonzero.
\end{enumerate}
\end{lem}

\begin{proof}
We begin with (a). 
By (\ref{epsiloniszero}) and (\ref{swap}),
\[
-2^{1-3n}\delta^1(c_4^n)=
\sum_{r=0}^n\binom nr 4^{-r}a_{r,n-r}
=\sum_{r=0}^{\left\lfloor\frac{n-1}2\right\rfloor}\binom nr
\left(4^{-r}-4^{r-n}\right)a_{r,n-r}
\]
and since
$a_{r,n-r}$ generates $\Z/(3^{\mu_3(n)+1})$
in $\coker h$,
it suffices to show that
the coefficient of $a_{r,n-r}$ is divisible
by $3n$ for all $n$.  But Lemma \ref{adic}(a)
implies
\[
4^{-r}-4^{r-n}=\frac{1-4^{2r-n}}{4^r}
\]
is a multiple of $3n-6r$, 
so it is enough
to show that $\binom nr(n-2r)$ is divisible 
by $n$.  This is clear if $r=0$.
For $0<r\leq\lfloor(n-1)/2\rfloor$,
\begin{align*}
\binom nr(n-2r)&=\frac nr\binom{n-1}{r-1}(n-2r)\\
&=n\left(\frac
{(n-1)(n-2)\cdots(n-r+1)}{(r-1)!}
\right)\left(
\frac{n-r}r-1
\right)\\
&=n\left(
\binom {n-1}{r}-\binom{n-1}{r-1}
\right).
\end{align*}

Next, we prove (b). 
Let $m=n+1$. 
Collecting terms 
in the formula from Lemma 
\ref{elliszero} (see Remark \ref{SORT}) yields
\[
\delta^1(c_4^{m-1}c_6)\doteq
\sum_{r=1}^{\left\lfloor\frac{m-1}2\right\rfloor}
\frac{4^{1-r}}r\binom{m-1}{r-1}(3r(1+2^{4r-2m})
-2m(1+2^{4r-2m-1}))b_{r,m-r}.
\]

Suppose first that $m\nequiv13\bmod 27$.  
Let $f(r,m)$ be the coefficient on $b_{r,m-r}$
in the above formula.  
Each $b_{r,m-r}$ generates
a copy of $\Z/(3^{\nu_3(2m+1)+1})
=\ZLOC/(6m+3)$ in $\coker h$, and 
the condition on $m$
implies $\nu_3(6m+3)\leq3$.  
Thus, to prove the first claim 
it will suffice to show that 
each $f(r,m)$ is a multiple of $3^3$
modulo $6m+3$.
This is true for $r=1$ since
\begin{align*}
f(1,m)&=3(1+2^{4-2m})-m(2+2^{4-2m})\\
&=3(1+32\cdot2^{-2m-1})-m(2+32\cdot2^{-2m-1})\\
&\equiv3(1-32)-m(2-32)\\
&\equiv-4\cdot3^3\bmod(6m+3).
\end{align*}
For $r>1$,
\begin{align*}
4^{r-1}f(r,m)&=\binom{m-1}{r-1}(3(1+2^{4r-2m}))
-\binom{m}r(2+2^{4r-2m})\\
&\equiv\binom{m-1}{r-1}(3(1-2^{4r+1}))
-\binom{m}r(2-2^{4r+1})\\
&\equiv\binom{m-1}{r-1}\left(
3(1-2^{4r+1})+\frac{1-2^{4r}}r\right)\bmod(6m+3).
\end{align*}
Let $A(r)=3(1-2^{4r+1})+(1-2^{4r})/r$. 
If $3$ does not divide $r$, then $A(r)$
is divisible by $3^3$, and so $f(r,m)b_{r,m-r}=0$
in those cases.  If $3$ divides $r$, then $A(r)$ is only
divisible by $3^2$, and this is sufficient
to annihilate $b_{r,m-r}$ except when $m\equiv4$
or $22\bmod27$.   However in those cases the binomial
coefficient $\binom{m-1}{r-1}$ contributes the
additional power of $3$ that is needed.   

Finally, if $m\equiv13\bmod27$, then 
$\nu_3(6m+3)>3$, so the calculation
of $f(1,m)$ above shows that $f(1,m)b_{1,m-1}\neq0$
in $\ZLOC/(6m+3)$. Hence $\delta^1(c_4^{m-1}c_6)$
is nonzero in $\coker h$.
\end{proof}
\begin{lem}\label{havecsix}
If $\ell\neq0$,
\[
\delta^1(c_4^nc_6\Delta^{\ell})\doteq
\begin{cases}
b_{\ell,n+2\ell+1}+\HOT, & \ell>0 \\
b_{2\ell,n+\ell+1}+\HOT, & \ell<0.
\end{cases}
\]
\end{lem}
\begin{proof}
Assume $\ell\neq0$ throughout.  For $-\ell\geq n+1$,
\begin{align*}
-2^{-4n+4\ell-1}\delta^1(c_4^nc_6\Delta^{\ell})&=
2^{4\ell+2}\sum_{r=2\ell+1}^{n+2\ell+1}\binom n{r-2\ell-1}
4^{-r}b_{r,n+3\ell+1-r}\\
&\hspace{.4truein}-2^{4\ell+3}\sum_{r=2\ell}^{n+2\ell}\binom n{r-2\ell}4^{-r}
b_{r,n+3\ell+1-r}
\end{align*}
with leading term $-8b_{2\ell,n+\ell+1}$.  For $\ell\geq n+1$,
\begin{align*}
-2^{-4n+4\ell-1}\delta^1(c_4^nc_6\Delta^{\ell})&=
-2^{-2n-2\ell}\sum_{r=\ell}^{n+\ell}\binom n{r-\ell}4^r
b_{r,n+3\ell+1-r}\\
&\hspace{.4truein}+2^{-2n-2\ell+1}
\sum_{r=\ell+1}^{n+\ell+1}\binom n{r-\ell-1}4^r
b_{r,n+3\ell+1-r}
\end{align*}
with leading term $-2^{-2n}b_{\ell,n+2\ell+1}$.

For $|\ell|<n+1$,
\begin{align*}
-2^{-4n+4\ell-1}\delta^1(c_4^nc_6\Delta^{\ell})&=
2^{4\ell+2}\sum_{r=2\ell+1}^{\left\lceil\frac{n+3\ell-1}2\right\rceil}\binom n{r-2\ell-1}
4^{-r}b_{r,n+3\ell+1-r}\\
&\hspace{.4truein}-2^{-2n-2\ell}\sum_{r=\ell}^{\left\lceil\frac{n+3\ell-1}2\right\rceil}
\binom n{r-\ell}4^r b_{r,n+3\ell+1-r}\\
&\hspace{.4truein}-2^{4\ell+3}\sum_{r=2\ell}^{\left\lceil\frac{n+3\ell-1}2\right\rceil}\binom n{r-2\ell}
4^{-r}b_{r,n+3\ell+1-r}\\
&\hspace{.4truein}+2^{-2n-2\ell+1}\sum_{r=\ell+1}^{\left\lceil\frac{n+3\ell-1}2\right\rceil}\binom n{r-\ell-1}
4^r b_{r,n+3\ell+1-r}
\end{align*}
with leading term $-2^{-2n}b_{\ell,n+2\ell+1}$ if $\ell>0$
and $-8b_{2\ell,n+\ell+1}$ if $\ell<0$.
\end{proof}
\begin{proof}[Proof of Proposition \ref{propcombo}]
Suppose first that $\epsilon=0$.  If $m<0$,
then $\LMAX<0$, and
\[
\delta^1(C^m_v)\doteq A^m_{\left\lfloor\frac{m-1}2
\right\rfloor-2\LMAX+2v}+\HOT
\]
for $v\geq0$, by Lemma \ref{nocsix}.  Thus, 
\begin{equation}\label{MATRIX1}
\delta^1\big|_{W_{0,m}}=
\left[
\begin{array}{cccc}
\vdots & \vdots & \, & \,  \\
u_0 & * & \, & \,  \\
0 & * & \vdots & \,  \\
\vdots & u_1 & * & \,  \\
\, & 0 & * & \,  \\
\, & \vdots & u_2 & \,  \\
\, & \, & 0 & \ddots  
\end{array}\right]
\end{equation}
where $u_0,u_1,u_2,\ldots\in\ZLOC^{\times}$
and $u_0$ is in the row corresponding to
$A^m_{\floor{\frac{m-1}2}-2\LMAX}$. 
By (\ref{MATRIX1}),
$\delta^1\big|_{W_{0,m}}$ has 
trivial kernel, and has cokernel generated by
\[\left\{
A^m_0,\ldots,A^m_{\left\lfloor\frac{m-1}2
\right\rfloor-2\LMAX-1}\right\}\cup\left\{
A^m_{\left\lfloor\frac{m-1}2
\right\rfloor-2\LMAX+i}:i\geq1,\rm{\, odd}
\right\}.
\]
Each $A_v^m$ generates
$\Z/(3^{\nu_3(m)+1})$ in $\coker h$.
This proves parts (a) and (b) for $m<0$.

If $m>0$, then $\LMAX\geq0$. By
Lemmas \ref{nocsix} and \ref{lemcombo}(a), 
\[
\delta^1(C^m_v)\doteq
\begin{cases}
A^m_{\left\lfloor\frac{m-1}2\right\rfloor-\LMAX+v}
+\HOT, & 0\leq v<\LMAX\\
0, & v=\LMAX\\
A^m_{\left\lfloor\frac{m-1}2\right\rfloor-2\LMAX+2v}
+\HOT, & v>\LMAX.
\end{cases}
\]
Thus,
\begin{equation}\label{MATRIX2}
\delta^1\big|_{W_{0,m}}=
\left[
\begin{array}{ccccc}
\vdots & \vdots & \, & \, & \,\\
u_0 & * & \, & \, & \,\\
0 & u_1 & \, & \vdots & {\pmb\vdots}\\
\vdots & 0 & \ddots & * & {\pmb 0} \\
\, & \vdots & \, & u_y & {\pmb 0}\\
\, & \, & \, & 0 & {\pmb 0}\\
\, & \, & \, & \vdots & {\pmb\vdots}\\
\end{array}
\right]\boxplus
\left[
\begin{array}{cccc}
\vdots & \vdots & \,& \,\\
u_{y+1} & * & \, & \,\\
0 & * & \vdots & \,\\
\vdots & u_{y+2} & *& \,\\
\, & 0 & * & \, \\
\, & \vdots & u_{y+3} & \, \\
\, & \, & 0 & \ddots
\end{array}
\right]
\end{equation}
where 
the $u_i$ are units in $\ZLOC$.
Here, $u_0$ is in the row
corresponding to $A^m_{\left\lfloor\frac{m-1}2\right\rfloor
-\LMAX}$, $u_y$
is in the row corresponding to $A^m_{\left\lfloor\frac{m-1}2\right\rfloor
-1}$, $u_{y+1}$ is in the row corresponding to
$A^m_{\left\lfloor\frac{m-1}2\right\rfloor+2}$,
and the zero column in bold corresponds to 
$C^m_{\LMAX}$. By (\ref{MATRIX2}),
$\delta^1\big|_{W_{0,m}}$
has kernel generated by $C^m_{\LMAX}$,
and has cokernel generated by
\[
\left\{
A^m_0,A^m_1,\ldots,A^m_{\left\lfloor\frac{m-1}2\right\rfloor
-\LMAX-1}
\right\}
\cup
\left\{
A^m_{\floor{\frac{m-1}2}}
\right\}\cup
\left\{
A^m_{\left\lfloor\frac{m-1}2\right\rfloor+i}:i\geq 1,\text{ odd}
\right\}.
\]
Since $C_v^m$ generates $\Z/(3^{\nu_3(m)+1})$
in $\coker g$ and $A^m_v$ generates $\Z/(3^{\nu_3(m)+1})$
in $\coker h$, this proves parts (a) and (b) for $m>0$.

Suppose next that $\epsilon=1$.  If $m\leq0$,
then $\LMAXB<0$.  By Lemma \ref{havecsix}, 
\[
\delta^1(D^{m}_v)\doteq B^{m}_{\left\lfloor\frac{m-1}2
\right\rfloor-2\LMAXB+2v}+\HOT
\]
for $v\geq0$.  
Thus,
$\delta^1\big|_{W_{1,m}}$
is represented by a matrix of the form
identical to (\ref{MATRIX1}),
where in this case the unit $u_0$ 
appears in the row corresponding to
$B^{m}_{\left\lfloor\frac{m-1}2\right\rfloor-2\LMAXB}$.  
Thus $\delta^1\big|_{W_{1,m}}$
has trivial kernel, and its cokernel
is generated by
\[
\left\{
B^{m}_0,\ldots,
B^{m}_{\floor{\frac{m-1}2}-2\LMAXB-1}
\right\}\cup\left\{
B^{m}_{\left\lfloor\frac{m-1}2
\right\rfloor-2\LMAXB+i}:i\geq1,\text{\ odd}
\right\}.
\]
Since each $B_v^m$ generates
$\Z/(3^{\mu_3(2m+1)+1})$ in $\coker h$, 
this proves part (c).

If $m>0$, then $\LMAXB\geq0$.  
As long as $m\nequiv13\bmod27$, 
\[
\delta^1(D^{m}_v)\doteq
\begin{cases}
B^{m}_{\left\lfloor\frac{m-1}2
\right\rfloor-\LMAXB+v}+\HOT, & 0\leq v<\LMAXB \\
0, & v=\LMAXB \\
B^{m}_{\left\lfloor\frac{m-1}2
\right\rfloor-2\LMAX+2v}+\HOT, & v>\LMAXB.
\end{cases}
\]
by Lemmas \ref{lemcombo}(b)
and \ref{havecsix}.
The matrix representation in 
this case is of the form identical
to (\ref{MATRIX2}) above, where in this case
$u_0$ is in the row corresponding to 
$B^{m}_{\left\lfloor\frac{m-1}2\right\rfloor
-\LMAXB}$, $u_y$
is in the row corresponding to $B^{m}_{\left\lfloor\frac{m-1}2\right\rfloor
-1}$, and $u_{y+1}$ 
is in the row corresponding to
$B^{m}_{\left\lfloor\frac{m-1}2\right\rfloor+2}$.
Thus $\delta^1\big|_{W_{1,m}}$
has kernel generated by $D^m_{\LMAXB}$,
and has cokernel generated by 
\[
\left\{
B^{m}_0,B^{m}_1,\ldots,B^{m}_{\left\lfloor\frac{m-1}2\right\rfloor
-\LMAXB-1}
\right\}\cup
\left\{
B^{m}_{\floor{\frac{m-1}2}}
\right\}\cup
\left\{
B^{m}_{\floor{\frac{m-1}2}+i}:
i\geq 1,\text{ odd}
\right\}.
\]
Since each $D_v^m$ generates
$\Z/(3^{\mu_3(2m+1)+1})$ in $\coker g$
and each $B^m_v$ generates
$\Z/(3^{\mu_3(2m+1)+1})$ in $\coker h$,
this proves part (d).

If $m>0$ and $m\equiv 13\bmod 27$,  
Lemma \ref{lemcombo}(b) implies
$\delta^1\big|_{W_{1,m}}$ has matrix
representation identical in form to
\eqref{MATRIX2} except for the column
in bold; it is not a 
column of zeros in this case.  Rather, 
it has at least one
nonzero entry in and above the row
containing $u_y$ by Lemma \ref{lemcombo}(a). 
This makes the kernel and cokernel
less straightforward to compute
(see Remark \ref{examplecase}).
What we can conclude, however, 
is that the cokernel
of $\delta^1\big|_{W_{1,m}}$
contains copies of $\Z/(3^{\mu_3(2m+1)+1})$
generated by
\[
\left\{
B^{m}_{\floor{\frac{m-1}2}}
\right\}\cup
\left\{
B^{m}_{\floor{\frac{m-1}2}+i}:
i\geq 1,\text{ odd}
\right\}
\]
which proves part (e).
\end{proof}
\begin{defn}\label{unknowns}
Let
\[
U_1:=\ker\left(
\delta^1\bigg|_{\displaystyle\bigoplus_{0<m\equiv13\bmod27}
W_{1,m}}\right)
\]
and define $U_2$ via the direct sum decomposition
\[
\coker\left(
\delta^1\bigg|_{\displaystyle\bigoplus_{0<m\equiv13\bmod27}
W_{1,m}}\right)=
\bigoplus_{n\in\N}\left(\bigoplus_{0<m\equiv13\bmod27}
\Z/(3^{\nu_3(6m+3)})
\right)\oplus U_2.
\]
\end{defn}
\begin{proof}[Proof of Theorem \ref{cohomology}]
By Propositions \ref{filtration} and \ref{delzero},
$H^0C^*=\ZLOC$.
Proposition \ref{delzero} also implies
$\coker\delta^0=
\bigoplus_{\N}\ZLOC$
concentrated in degree zero.
By Proposition \ref{delone},
the degree zero part
of $\ker\delta^1$ is a copy
of $\ZLOC$ generated by $1_{MF}$.
Thus, the short exact sequence
(\ref{hone}) in Proposition 
\ref{filtration}
implies that $H^1C^*$ is a countable direct sum
of copies of $\ZLOC$ in degree zero, and 
is isomorphic to $\ker\delta^1$ in positive
degrees.  The result for $H^1C^*$ then follows from 
Proposition \ref{propcombo}.
The result for $H^2C^*$
follows from Propositions \ref{filtration}
and \ref{propcombo} and Definition \ref{unknowns}.
\end{proof}

\section{Differential on the $E_2$-page}\label{deetwodiff}

In this section we compute 
$\widetilde d:\Ext^1\to\coker\Psi$ 
and prove Theorem \ref{dtwodiff}.
The computation of $\widetilde d$ amounts to
a diagram chase with the maps
\[
\xymatrix{
\Gamma\ar[r]^-{\Phi} & \Gamma & \, \\
\, & B\oplus B\ar[r]^-{\Psi}\ar[u]^{-d\oplus0} & B
}
\]
of $C^{*,*}$.
Since
\[
\Ext^1=\ZLOC\{\Delta^k\alpha:k\in\mathbb Z\}
=\bigoplus_{k\in\Z}\Z/(3)
\]
where $\alpha$ is represented by $r\in\Gamma$
\cite{Bauer}, it suffices to compute 
$\widetilde d(\Delta^k\alpha)$ 
for all $k\in\mathbb Z$.

Consider first the case $k=0$.  
The element $r\in\Gamma$
is mapped to $3r$ under $\Phi$,
which in turn 
must be hit by some
element $y\in B\oplus B$ under $-d\oplus0$;
in fact $y=(-q_2,q_2)$ works.  Since
\[
\Psi(y)=(\psi_d+1)(-q_2)-\phi_f(q_2)=q_2-q_2=0,
\]
$\widetilde d(\alpha)=0$.

Next, suppose $k>0$.  The element 
$\Delta^k\alpha\in\Ext^1$
is represented by 
$\Delta^k r\in\Gamma$.  Under $\Phi$,
$\Delta^k r$ maps to 
$(2^{12k+2}-1)\Delta^k r$, which in turn
is the image of an element in $B\oplus B$ under
$-d\oplus0$, namely
\[
\left(-d\oplus0\right)\left(
\frac{1-2^{12k+2}}3\Delta^k q_2,0
\right)=(2^{12k+2}-1)\Delta^k r.
\]
By Lemma \ref{adic}(b),
$(1-2^{12k+2})/3\in\ZLOC^{\times}$. 
Thus, applying
$\Psi$ yields
\begin{align*}
\Psi\left(
\frac{1-2^{12k+2}}3\Delta^k q_2,0
\right)&=-\phi_f\left(
\frac{1-2^{12k+2}}3\Delta^k q_2
\right)\\
&=\frac{2^{12k+2}-1}3 q_4^{2k}\mu^k q_2\\
&\doteq b_{k,2k}-\overline b_{k,2k}
\end{align*}
and $b_{k,2k}-\overline b_{k,2k}$
represents the class 
$b_{k,2k}= 
B^{3k}_{\floor{\frac{3k-1}2}-k}\in\coker\Psi$.
Thus 
\begin{equation*} 
\widetilde d(\Delta^k\alpha)\doteq
B^{3k}_{\floor{\frac{3k-1}2}-k}. 
\end{equation*}
This class is nontrivial by
Proposition \ref{propcombo}. 

If $k<0$, a similar argument shows 
\[
\widetilde d(\Delta^k\alpha)\doteq
B^{3k}_{\floor{\frac{3k-1}2}-2k},
\]
also nontrivial by Proposition \ref{propcombo}.

The preceding arguments show that
\[\ker\widetilde d=\Z/(3)\] generated by $\alpha$, and that
$\operatorname{im}\widetilde d$ is generated
by 
\[
\left\{
B^{3k}_{\floor{\frac{3k-1}2}-k}:k>0\right\}
\cup\left\{B^{3k}_{\floor{\frac{3k-1}2}-2k}
:k<0\right\}.
\]
Since each $B_v^{m}$ generates 
$\Z/(3^{\nu_3(6m+3)})$ in $\coker\Psi$,
this proves Theorem \ref{dtwodiff}.

\section{Detection of Greek letter elements}\label{Greek}

In this section we use Theorem \ref{main} to
give evidence for Conjecture \ref{conjecture}.
We set up the discussion by first considering 
the algebraic alpha family element
\[
\alpha_1:=\alpha^a_{1/1}\in
\Ext^{1,4}_{BP_*BP}(BP_*,BP_*)
\]
of order 3 \cite{Rav:MU}. 
\begin{prop}\label{alpharesult}
$\alpha_1$ is detected by $\alpha=r\in\ker\widetilde d$.
\end{prop}
\begin{proof}
Note first that $\alpha$ is in the correct
bidegree.  Since the double complex bidegree
of any element in $\ker\widetilde d$ is $(0,1)$
and $\deg(r)=2$ in $\Gamma$, the
bidegree of $\alpha$ in 
the Adams-Novikov $E_2$-term for $Q(2)$ 
is $(s,t)=(0+1,2)$. But recall that the $E_2$-term for $Q(2)$ 
is indexed so that $E_2^{s,t}\Rightarrow\pi_{2t-s}Q(2)$
(see Eq.\ \eqref{spectralsequence}), so the
corresponding bidegree in $\Ext^{*,*}_{BP_*BP}(BP_*,BP_*)$
is $(s,2t)=(1,4)$.

In fact we know $\alpha$
must detect $\alpha_1$ because $\alpha_1$
is detected by $TMF$ by this same element 
$\alpha=r\in\Gamma$ \cite{Bauer}, 
and the diagram
of $E_2$-terms induced by
\begin{equation}\label{tmfdiag}
\xymatrix{
Q(2)\ar[r] & TMF\\
\, & L_{2}\threesphere\ar[u]\ar[ul]
}
\end{equation}
commutes. 
\end{proof}

Although we do not conjecture that 
the Adams-Novikov $E_2$-term for $Q(2)$ detects
the entire algebraic divided alpha family
\[
\{\alpha^a_{i/j}\in\Ext^{1,4i}_{BP_*BP}(BP_*,BP_*):0<i\in\Z,
j=\nu_3(i)+1\}
\]
(where $\alpha^a_{i/j}$ has additive order $3^j$),
it does contain elements
of the appropriate bidegrees and additive
orders that {\em could}
collectively detect it.  These
elements (other than $\alpha$ discussed above) 
live not in $\ker\widetilde d$, but
rather in $H^1C^*$.  For example,
$\alpha_2:=\alpha^a_{2/1}$ could be detected
by the class $C_0^1=c_4\in H^1C^*$, an element of order
3 in bidegree $(s,t)=(1,4)$.  Another example
is $\alpha^a_{3/2}$,
which could be detected by
$D_0^1=c_6\in H^1C^*$, a class in bidegree $(s,t)=(1,6)$ 
of order 9.  
In general, the candidate element
for detecting $\alpha^a_{i/j}$ is given by
\[
\begin{cases}
C^{i/2}_{\ell_0^{i/2}},&\text{if }i\text{ even},\\
\, &\, \\
D^{(i-1)/2}_{\ell_1^{(i-1)/2}},&\text{if }i\text{ odd}.
\end{cases}
\]
%

We now turn to the algebraic divided beta family.
Consider first the element
\[
\beta_1:=\beta^a_{1/1,1}\in
\Ext^{2,12}_{BP_*BP}(BP_*,BP_*)
\]
of order 3.  Like $\alpha_1$, 
we know this element has a nontrivial
target in the Adams-Novikov $E_2$-term
for $Q(2)$.  
\begin{prop}
$\beta_1$ is detected by $\beta:=
r^2\otimes r-r\otimes r^2\in\Ext^2$.
\end{prop}
\begin{proof}
As $\beta_1$ is on the 2-line,
$\beta$ must live in
$\Ext^2$ with double complex
bidegree $(0,2)$.  Using the fact that
 $\deg(r^2\otimes r-r\otimes r^2)=6$
in $\Gamma\otimes\Gamma$, an argument
analogous to the proof of Proposition
\ref{alpharesult} yields the result \cite{Bauer}.
\end{proof}
We now offer three further examples 
of beta elements and candidate elements 
for detecting them
in the Adams-Novikov $E_2$-term for $Q(2)$.
\begin{expl}\label{beta63}
The algebraic beta element
\[
\beta_{6/3}:=\beta^a_{6/3,1}\in\Ext^{2,84}_{BP_*BP}(BP_*,BP_*),
\]
itself an element of order 3,
is known to be a permanent cycle and 
represents a nontrivial homotopy element 
$\beta_{6/3,1}^h\in\pi_{82}\threesphere$.  
Any nontrivial target in the Adams-Novikov
$E_2$-term for $Q(2)$ must be in
bidegree $(s,t)=(2,42)$.
While there are no such elements in $\Ext^1$, 
the element
\[
\Delta^3\beta\in\Ext^2
\]
does have both the required bidegree
and additive order to detect $\beta_{6/3}$
\cite{Bauer}.
By the proof of Proposition \ref{propcombo}(d),
the other possible targets for this element in the
$E_2$-term for $Q(2)$ are the classes
\[
3B^{10}_0; 3B^{10}_4; 3B^{10}_5, 3B^{10}_7,
3B^{10}_9,\ldots\in\coker\widetilde d.
\]  
Each $B^{10}_v$ is multiplied by 3
because it generates $\Z/(3^{\nu_3(6\cdot10+3)})=\Z/(9)$.  
\end{expl}
\begin{expl}\label{beta99}
The Kervaire invariant problem at the prime 3
asks whether the classes
\[
\theta_j:=\beta_{3^{j-1}/3^{j-1},1}\in
\Ext^{2,4\cdot 3^j}_{BP_*BP}(BP_*,BP_*),
\quad j\geq1,
\]
the so-called {\em Kervaire classes}, are permanent
cycles.  For $j=3$, the corresponding Kervaire
class is $\theta_3=\beta_{9/9,1}$ in bidegree $(2,108)$,
and living in this same bidegree is 
$\beta_7:=\beta_{7/1,1}$ (itself not a Kervaire
class).  As in Example \ref{beta63}, there are no
elements in $\Ext^1$ to detect either
of these classes.  But there is a
candidate element in $\Ext^2$; in this case it is
$\Delta^4\beta$.
Potential detecting elements in 
$\coker\widetilde d$ for $\theta_3$ and $\beta_7$  
were identified
in Remark \ref{examplecase} from Section \ref{Intro}:
namely, any element in
\begin{align*}
\left(\Z/(81)\{B_0^{13},B_1^{13}\}/
(B_0^{13}-3B_1^{13}=0)\right)
&\oplus\Z/(81)\{B_6^{13}\}\\
&\oplus\Z/(81)\{B_7^{13},B_9^{13},B_{11}^{13},\ldots\}
\end{align*}
suitably multiplied by a power of 3 so as to
obtain an element of order 3.  Along with $\Delta^4\beta$,
these 
are the only elements in $E_2^{2,54}Q(2)$.  
\end{expl}
\begin{expl}
In our final example we consider the
algebraic beta element
\[
\beta^a_{9/3,2}\in\Ext^{2,132}_{BP_*BP}(BP_*,BP_*),
\]
an element of order 9.  It is the
beta element $\beta^a_{i/j,k}$ of lowest
topological degree (which in this case
is $t-s=130$) with $k>1$.  Since there
are no elements in $\Ext^2$ or $\Ext^1$
capable of detecting $\beta^a_{9/3,2}$ for
degree reasons,
we look in $\coker\widetilde d$.
By Proposition \ref{propcombo}(d)
there are indeed candidate detecting elements
in $\coker\widetilde d$ given by
\[
B_0^{16}; B_1^{16}; B_7^{16}; B_8^{16}, B_{10}^{16}, B_{12}^{16},\ldots,
\]
themselves classes of order 9
in $E_2^{2,66}Q(2)$.
\end{expl}

\footnotesize
\bibliographystyle{plain}
\bibliography{math}

\end{document}